\documentclass[12pt, reqno]{amsart}
\usepackage{amsmath, amssymb, stmaryrd, setspace, nicefrac, multicol, amsthm, geometry}
  \usepackage[pdftex,
  bookmarks = false,
  pdfstartview = FitBH,
  linktocpage = true,
  pagebackref = true,
  pdfhighlight = /O,
  colorlinks=true,
  linkcolor=blue,
  citecolor=blue, 
  filecolor = blue,
  urlcolor = blue,
  menucolor = blue,
]{hyperref}

\newtheorem{thm}{Theorem}
\newtheorem{lem}[thm]{Lemma}
\newtheorem{cor}[thm]{Corollary}
\newtheorem{prop}[thm]{Proposition}

\theoremstyle{definition}
\newtheorem{defn}[thm]{Definition}
\newtheorem{Q}[thm]{Question}

\newtheorem{ex}[thm]{Example}

\newcommand{\uh}{{\upharpoonright}}
\newcommand{\cs}{2^\omega}
\newcommand{\str}{2^{<\omega}}
\newcommand{\MLR}{\mathsf{MLR}}
\newcommand{\SR}{\mathsf{SR}}
\newcommand{\KR}{\mathsf{KR}}
\newcommand{\dom}{\mathrm{dom}}

\newcommand{\oi}{\mathrm{OI}}
\newcommand{\avg}{\mathit{Avg}}
\newcommand{\rate}{\mathit{Rate}}

\newcommand{\llb}{\llbracket}
\newcommand{\rrb}{\rrbracket}

\renewcommand{\P}{\mathcal{P}}
\newcommand{\fr}{^{\frown}}
\renewcommand{\tt}{\mathit{tt}}
\newcommand{\U}{\mathcal{U}}

\newcommand{\halts}{{\downarrow}}

\usepackage{soulutf8}

\title{Randomness Extraction in Computability Theory}
\author{Douglas Cenzer and Christopher P. Porter}
\date{\today} 
\begin{document}

\maketitle 
\singlespacing

\begin{abstract}
In this article, we study a notion of the extraction rate of Turing functionals that translate between notions of randomness with respect to different underlying probability measures. We analyze several classes of extraction procedures: a first class that generalizes von Neumann's trick for extracting unbiased randomness from the tosses of a biased coin, a second class based on work of generating biased randomness from unbiased randomness by Knuth and Yao, and a third class independently developed by Levin and Kautz that generalizes the data compression technique of arithmetic coding. For the first two classes of extraction procedures, we identify a level of algorithmic randomness for an input that guarantees that we attain the extraction rate along that input, while for the third class, we calculate the rate attained along sufficiently random input sequences.
\end{abstract}

\section{Introduction}

The aim of this study is to analyze the rate of the extraction of randomness via various effective procedures using the tools of computability theory and algorithmic randomness.  Our starting point is a classic problem posed by von Neumann in \cite{Von51}, namely that of extracting unbiased randomness from the tosses of a biased coin.  Von Neumann provides an elegant solution to the problem: Toss the biased coin twice.  If the outcome is $\mathit{HH}$ or $TT$, then discard these tosses. Otherwise, if the outcome is $HT$, then output {\bf H}, and if the outcome is $TH$, then output {\bf T}. Notice in the case that the coin comes up heads with probability $p$,
\begin{itemize}
\item the probability of $\mathit{HH}$ is $p^2$, 
\item the probability of $\mathit{TT}$ is $(1-p)^2$, and
\item the probability of $\mathit{HT}$ (and that of $\mathit{TH}$) is $p(1-p)$.  
\end{itemize}
It follows from the independence of the events $H$ and $T$ that with probability one the derived sequence will be an infinite sequence in which the events {\bf H} and {\bf T} each occur with probability 1/2.

It is well known that von Neumann's procedure is rather inefficient, since on average $\frac{1}{p(1-p)}$ biased bits are required to produce one unbiased bit when the biased coin comes up heads with probability $p\in(0,1)$.  For instance, in the case that $p=\frac{1}{2}$, where we are given a fair coin to begin with, four tosses on average yield one bit of output (a rate that is four times the rate attained simply by reading off the tosses of the coin).
However, a number of improvements have been found. For instance, in \cite{Per92}, Peres studies a sequence of procedures obtained by iterating von Neumann's procedure and calculates the associated extraction rate of each such procedure. As defined by Peres, given a monotone function $\phi:\str\rightarrow\str$, the extraction rate of $\phi$ with respect to the bias $p$ is defined to be
\[
\limsup_{n\rightarrow\infty}\dfrac{E(|\phi(x_1,x_2,\dotsc,x_n)|)}{n},
\]
where the bits $x_i$ are independent and $(p,1-p)$-distributed and $E$ stands for expected value (with respect to the $p$-Bernoulli measure on $\cs$).  Setting $(\phi_k)_{k\in\omega}$ to be  the sequence of procedures defined by Peres, he proves that, when tossed a coin that comes up heads with probability $p\in(0,1)$
\[
\lim_{k\rightarrow\infty}\limsup_{n\rightarrow\infty}\dfrac{E(|\phi_k(x_1,x_2,\dotsc,x_n)|)}{n}=H(p),
\]
where $H(p)=-p\log(p)-(1-p)\log(1-p)$ is the entropy associated with the underlying source.

This notion of an extraction rate of an effective procedure has not been thoroughly studied from the point of view of computability theory (however, see Doty \cite{Doty08} and Toska \cite{Toska14}, each of which study a more local notion of rate of certain procedures applied to specific inputs).  In this article, we study a definition of the extraction rate for Turing functionals that accept their input with probability one (referred to as  \emph{almost total} functionals).  In particular, we can formalize certain randomness extraction procedures as Turing functionals and study the behavior of these functionals when applied to algorithmically random sequences.  For a number of such functionals, it is known that almost every sequence attains the extraction rate; here we provide a sufficient level of algorithmic randomness that guarantees this result. 

We consider three main examples here:
\begin{enumerate}
\item functionals defined in terms of maps on $\str$ that we call block maps, which generalize von Neumann's procedure,

\medskip

\item functionals derived from certain trees called discrete distribution generating trees (or DDG trees, for short), introduced by Knuth and Yao \cite{KnuYao76} in the study of non-uniform random number generation, and

\medskip

\item a procedure independently developed by Levin \cite{LevZvo70} and Kautz \cite{Kau91} for converting biased random sequences into unbiased random sequences.
\end{enumerate}
Notably, our analysis of the extraction rates of these three classes of examples draws upon the machinery of effective ergodic theory, using certain effective versions of Birkhoff's ergodic theorem (and, in the case of the Levin-Kautz procedure, an effective version of the Shannon-McMillan-Breiman theorem from classical information theory due to Hoyrup \cite{Hoy12}). 

The remainder of the paper is as follows.  In Section \ref{sec-bg}, we lay out the requisite background for this study.  Next, in Section \ref{sec-er}, we formally define the extraction rate of a Turing functional, derive several preliminary results, and introduce several basic examples.  We then turn to more general examples:  Turing functionals derived from block maps in Section \ref{sec-block}, Turing functionals derived from computable DDG trees in Section \ref{sec-ddg}, and the Levin-Kautz procedure in Section \ref{sec-lk}.  We conclude with several open questions in Section \ref{sec-questions}.

\section{Background}\label{sec-bg}

\subsection{Notation}
The set of finite binary strings will be written as $\str$; members of $\str$ will be written as lowercase Greek letters, $\sigma,\tau,\rho$, and so on. The set of infinite  binary sequences will be written as $\cs$; members of $\cs$ will be written as uppercase Roman letters $X,Y,Z$.  For a finite string $\sigma
\in \str$, let $|\sigma|$ denote the length of $n$.  For two strings $\sigma,\tau$, say that $\tau$ \emph{extends} $\sigma$
and write $\sigma \preceq \tau$ if $|\sigma| \leq |\tau|$ and $\sigma(i)
= \tau(i)$ for $i < |\sigma|$. For $X \in \cs$, $\sigma \prec X$ means
that $\sigma(i) = X(i)$ for $i < |\sigma|$. Let $\sigma^{\frown} \tau$
denote the concatenation of $\sigma,\tau\in\str$; we similarly define the concatenation $\sigma^{\frown}X$ of $\sigma\in\str$ and $X\in\cs$. Let $X
\uh n$ denote the string $\sigma\prec X$ with $|\sigma|=n$.  For $n<m$, $X\uh[n,m)$ denotes the string $X(n)\dotsc X(m-1)$. The empty string will be written as $\epsilon$.

Two sequences $X,Y\in\cs$ may be coded together into $Z = X \oplus Y$, where 
$Z(2n) = X(n)$ and $Z(2n+1) =
Y(n)$ for all $n$.  For a finite string $\sigma$, let $\llb\sigma\rrb$ denote $\{X \in \cs:
 \sigma \prec  X\}$. We shall refer to $\llb\sigma\rrb$ as the \emph{cylinder}
 determined by $\sigma$. Each such interval is a clopen set and the
 clopen sets are just finite unions of intervals.
 
 \subsection{Trees}
A nonempty closed set $P\subseteq\cs$ may be identified with a tree $T_P
 \subseteq \str$ where $T_P = \{\sigma: P \cap \llb\sigma\rrb \neq
 \emptyset\}$. Note that $T_P$ has no dead ends. That is, if $\sigma
 \in T_P$, then either $\sigma^{\frown}0 \in T_P$ or $\sigma^{\frown}1
 \in T_P$ (or both).
For an arbitrary tree $T \subseteq \str$, let $[T]$ denote the
set of infinite paths through $T$; that is,  $[T]=\{X\in\cs\colon (\forall n)\;X\uh n\in T\}$.
It is well-known that $P \subseteq \cs$ is a closed set if and only if
$P = [T]$ for some tree $T$.  $P$ is a \emph{$\Pi^0_1$ class}, or an \emph{effectively
closed set}, if $P = [T]$ for some computable tree $T$.

\subsection{Turing functionals}
 Recall that a continuous function $\Phi: \cs \to \cs$ may be defined from a function $\phi: \str \to \str$, which we refer to as a \emph{generator} of $\Phi$, satisfying the conditions

\begin{enumerate}
\item[(i)] For $\sigma,\tau\in\str$, if $\sigma \preceq \tau$, then $\phi(\sigma) \preceq \phi(\tau)$.

\item[(ii)] For all $X \in \cs$, $\lim_{n\rightarrow\infty} |\phi(X \uh n)| = \infty$.
\end{enumerate}

\noindent Note by the compactness of $\cs$, a generator $\phi$ for a continuous function $\Phi$ satisfies the condition:

\begin{enumerate}
\item[(iii)] For all $\sigma\in\str$ and $m\in\omega$, there exists $n\in\omega$ such that for every $\sigma \in \{0,1\}^n$, $|\phi(\sigma)| \geq m$. 
\end{enumerate}

\noindent We then have $\Phi(X) = \bigcup_n \phi(X \uh n)$. The total  Turing functionals $\Phi: \cs \to \cs$ are those which may be defined in this manner from a computable generator $\phi: \str \to \str$.  We will sometimes refer to total Turing functionals as \emph{$\tt$-functionals.} The partial Turing functionals $\Phi:\subseteq\cs\rightarrow\cs$ are given by those $\phi: \str \to \str$ which only satisfy condition (i) (we will still refer to such functions as generators). In this case $\Phi(X) = \bigcup_n \phi(X \uh n)$
may be only a finite string. 
 
 We set $\dom(\Phi) = \{X: \Phi(X) \in \cs\}$.  For $\tau \in \str$ we also define
\[
\Phi^{-1}(\tau)=\{ \sigma\in \str : \tau \preceq \phi(\sigma)\;\&\;(\forall \sigma'\prec\sigma)\; \tau\not\preceq\phi(\sigma')\}.  
\]
In particular, by our above convention, we have $\Phi^{-1}(\epsilon)=\{\epsilon\}$. Similarly, for $S \subseteq \str$ we define $\Phi^{-1}(S) = \bigcup_{\tau \in S} \Phi^{-1}(\tau)$. For $\mathcal{A}\subseteq\cs$, we denote by $\Phi^{-1}(\mathcal{A})$ the set $\{X\in \dom(\Phi):\Phi(X) \in \mathcal{A}\}$. Note in particular that $\Phi^{-1}(\llb\tau\rrb) = \llb\Phi^{-1}(\tau)\rrb \cap \dom(\Phi)$.

\subsection{Computable measures on $\cs$}

Recall that a measure $\mu$ on $\cs$ is computable if there is a computable function $f:2^{<\omega}\times\omega\to\mathbb{Q}_2$ such that
$|\mu(\llbracket\sigma\rrbracket)-f(\sigma,i)|\leq 2^{-i}$.  For a prefix-free $V\subseteq\str$ (i.e., for $\sigma\in V$, if $\sigma\prec\tau$, then $\tau\notin V$), we set $\mu(\llb V\rrb)=\sum_{\sigma\in V}\mu(\sigma)$.  Hereafter, we will write $\mu(\llbracket\sigma\rrbracket)$ as $\mu(\sigma)$ for strings $\sigma$ and $\mu(\llbracket V\rrbracket)$ as $\mu(V)$ for $V\subseteq 2^{<\omega}$.  We also denote the \textit{Lebesgue measure} by $\lambda$, where $\lambda(\sigma)=2^{-|\sigma|}$ for $\sigma\in\str$.

\subsection{Notions of algorithmic randomness}

We assume that the reader is familiar with the basics of algorithmic randomness; see, for instance \cite{Nie09}, \cite{DowHir10}, \cite{SheUspVer17}, or the more recent \cite{FraPor20}.
Let $\mu$ be a computable measure on $\cs$.  Recall that a \emph{$\mu$-Martin-L\"of test} is a sequence $(\mathcal{U}_i)_{i\in\omega}$ of uniformly effectively open subsets of $\cs$ such that for each $i$,
\[
\mu(\mathcal{U}_i)\leq 2^{-i}.
\]
Moreover, $X\in\cs$ passes the $\mu$-Martin-L\"of test $(\mathcal{U}_i)_{i\in\omega}$ if $X\notin\bigcap_{i \in \omega}\mathcal{U}_i$.  Lastly, $X\in\cs$ is \emph{$\mu$-Martin-L\"of random}, denoted $X\in\MLR_\mu$, if $X$ passes every $\mu$-Martin-L\"of test. When $\mu$ is the Lebesgue measure $\lambda$, we often abbreviate $\MLR_\mu$ by $\MLR$. 

We can obtain alternative notions of randomness by modifying the definition of a Martin-L\"of test.  We will work with two such alternatives in this paper. Let $\mu$ be a computable measure on $\cs$ and  $X\in\cs$. 
\begin{itemize}
\item[(i)] $X$ is $\mu$-\emph{Schnorr random} (written $X\in\SR_\mu$) if and only if $X$ is not contained in any $\mu$-Martin-L\"{o}f test $(\mathcal{U}_i)_{i\in\omega}$ with the additional condition that $\mu(\U_i)$ is computable uniformly in $i$.

\item[(ii)] $X$ is $\mu$-\emph{Kurtz random} (written $X\in\KR_\mu$) if and only if $X$ is not contained in any $\Pi^0_1$ class of $\mu$-measure~$0$ (equivalently, if and only if it is not contained in any $\Sigma^0_2$ class of $\mu$-measure~$0$).
\end{itemize}
Note that $\MLR_\mu\subseteq\SR_\mu\subseteq\KR_\mu$ for every computable measure $\mu$.

We are particularly interested in the interaction between Turing functionals and computable measures on $\cs$.  
For computable measure $\mu$ on $\cs$, a Turing functional $\Phi:\cs\rightarrow\cs$ is \emph{$\mu$-almost total} if $\mu(\dom(\Phi))=1$.

\begin{lem}
A Turing functional $\Phi$ is $\mu$-almost total if and only if $\KR_\mu\subseteq\dom(\Phi)$.
\end{lem}

\begin{proof}
If $\KR_\mu\subseteq\dom(\Phi)$, then clearly $\Phi$ is $\mu$-almost total.  For the other direction, observe that $\dom(\Phi)$ is a $\Pi^0_2$ subset of $\cs$.  Thus, if $\Phi$ is $\mu$-almost total, it follows that $\cs\setminus\dom(\Phi)$ is a $\Sigma^0_2$ $\mu$-nullset, and hence $\cs\setminus\dom(\Phi)=\bigcup_{i\in\omega}\mathcal{U}_i$ where each $\mathcal{U}_i$ is a $\Pi^0_1$ $\mu$-nullset.  Thus if $X\notin\dom(\Phi)$, there is some $i$ such that $X\in\mathcal{U}_i$, so $X$ cannot be $\mu$-Kurtz random.
\end{proof}

\section{Extraction Rates}\label{sec-er}

\subsection{The definition of extraction rate via a generator} We are interested in a version of the \emph{use function} of a Turing functional $\Phi$ which arises from a given generator $\phi$.
Let $u_{\phi}(X,n)$ be the least $m$ such that $|\phi(X \uh m)| \geq n$.  Then the \emph{extraction rate} of the computation of $Y = \Phi(X)$ from $X$ is given by the ratio 
\[
\frac{n}{u_{\phi}(X,n)},
\] 
that is, the relative amount of input from $X$ needed to compute the first $n$ values of $Y$.

There is an alternative definition which is more straightforward. The $\phi$-\emph{output/input ratio} of $\sigma$, $\oi_\phi(\sigma)$, is defined to be
\[\oi_\phi(\sigma)=\frac{|\phi(\sigma)|}{|\sigma|}.\]

\begin{lem} For any  Turing functional $\Phi$ with generator $\phi$ and any $X\in\cs$ such that $\Phi(X)\in\cs$,
\[
\lim_{n\rightarrow\infty} \dfrac{|\phi(X \uh n)|}{n} = \lim_{m\rightarrow\infty} \dfrac{m}{u_{\phi}(X,m)},
\]
provided that both limits exists. 
\end{lem}

\begin{proof} Fix an input $X$.  Let $m_0 = 0$ and let $m_{k+1}$ be the least $m>m_k$ such that 
$|\phi(X \uh m)| > |\phi(X \uh m_k)|$. Let $n_k = |\phi(X\uh m_k)|$.  Then for each $k>0$, $u_{\phi}(X,n_k) = m_k$ and hence
\[
\oi_{\phi}(X \uh m_k) = \frac {n_k}{m_k} = \frac {n_k}{u_{\phi}(X,n_k)},
\]
so that the two sequences have identical infinite subsequences, and hence the limits must be equal (since they are assumed to exist).
\end{proof}

Let us write $\oi_{\phi}(X)$ for $\limsup_{n\rightarrow\infty} \oi_{\phi}(X \uh n)$; we refer to this as the \emph{$\phi$-extraction rate along} $X$.  For the specific examples of extraction rates that we calculate in the remaining sections, we will work with specific generators defined from the randomness extraction literature. 

\subsection{Canonical generators}
Suppose that we are given an almost total Turing functional and would like to determine its extraction rate.  Which generator should we use?  For instance, we would like to say that the extraction rate for a constant function should be very low and should approach 0 in the limit. However, consider the following example.

\begin{ex} \label{ex1} Let $\Phi(X) = 0^{\omega}$ for all $X \in \cs$ and let $\phi(\sigma) = 0^{|\sigma|}$ for all $\sigma \in \str$. Then $u_{\phi}(X,n) = n$ for all $n$ and thus $\lim_{n\rightarrow\infty}\frac{n}{u_{\phi}(X,n)} = 1$ for all $X$.  
\end{ex}

To avoid this problem, we can work with a canonical generator of a Turing functional, which may be defined as follows. 

\begin{defn} For any partial continuous function $\Phi$, the \emph{canonical generator} $\phi$ for $\Phi$ is defined by letting $\phi(\sigma)$ 
be the longest common initial segment of all members of $\{\Phi(X): \sigma \prec X\}$. 
\end{defn} 
 

\begin{ex}{\ }
\begin{itemize}
\item[(i)] The identity function $\phi$ on strings is the canonical generator of the identity function on $\cs$ and thus for  $X\in\cs$, the use $u_{\phi}(X,n) = n$ for all $n\in\omega$, so that $\lim_{n\rightarrow\infty} \frac{n}{u_{\phi}(X,n)} = 1$. 

\item[(ii)] If $\Phi(X) = X \oplus X$, then its canonical generator is given by $\phi(\sigma) = \sigma \oplus \sigma$ for $\sigma\in\str$ (where $\sigma\oplus\sigma$ is the finite string defined as in the infinite case). Thus for $X\in\cs$, $\lim_{n\rightarrow\infty} \frac{n}{u_{\phi}(X,n)} = \frac12$.
\end{itemize}
\end{ex}
 
Note that if $\phi$ is the canonical generator for a constant function $\Phi(X) = C$, then we have $\phi(\sigma) = C$, an infinite sequence, for every $\sigma$. To avoid this unpleasantness, we can further restrict our functions to the non-constant functions.

\begin{defn} A partial continuous function $\Phi$ is \emph{nowhere constant} if for any string $\sigma$, either $\Phi(X)$ is undefined (that is, it is a finite string) for some $X \in \llb \sigma \rrb$, or there exist $X_1 \neq X_2$ in $\llb \sigma \rrb$ such that $\Phi(X_1) \neq \Phi(X_2)$. 
\end{defn}

It is easy to see that if $\Phi$ is nowhere constant, then the canonical generator is a well-defined map taking strings to strings and satisfies condition (i) in the definition of a representative of a functional. Moreover, the canonical generator of a functional has the following nice property, which is immediate from the definition.



\begin{lem} Let $\Phi$ be a partial continuous functional on $\cs$ with canonical generator $\phi: \str \to \str$.  Then for all $\sigma$ such that $\sigma0,\sigma1\in\dom(\phi)$, if $\phi(\sigma0)\succeq\tau$ and $\phi(\sigma1)\succeq\tau$, then $\phi(\sigma)\succeq\tau$.

\end{lem}

Next we consider the computability of the canonical represenation.  

\begin{prop} \label{prop1} If $\Phi$ is a total, nowhere constant Turing functional, then the canonical generator $\phi$ of $\Phi$ is computable. 
\end{prop}  

\begin{proof}  Let $\Phi: \cs \to \cs$ be a total, nowhere constant Turing functional and let $\psi: \str \to \str$ be some computable generator of $\Phi$. Then we can compute, for each $m$, a value $n_m$ such that $|\psi(\sigma)| \geq m$ for all strings $\sigma$ of length $\geq n_m$.  Now let a string $\sigma$ be given.  Since $\Phi$ is nowhere constant, we can compute a least value $m$ such that there exist $\tau_0,\tau_1\succeq \sigma$ such that $|\psi(\tau_i)|\geq m$ for $i\in\{0,1\}$ and $\psi(\tau_0)\neq\psi(\tau_1)$.
  Then the value $\phi(\sigma)$ of the canonical generator can be computed by letting $\phi(\sigma)$ be the common initial segment $\psi(\sigma') \uh m$ for all $\sigma'$ of length $n_m$ extending $\sigma$. 
\end{proof}

On the other hand, if $\Phi$ is only a partial computable, nowhere constant function, then the canonical generator of $\Phi$ need not be computable. 

\begin{ex} \label{ex2} Let $E$ be some noncomputable c.e.\ set and define $\Phi(0^n 1 X) = X$ if $n \in E$ and undefined otherwise. 
Then for the canonical generator $\phi$ of $\Phi$, we have $\phi(0^n 1 \sigma) = \sigma$ if $n\in E$ and $\phi(0^n1 \sigma) = \epsilon$, otherwise.  We can modify this to get an almost total functional by letting $\Phi(0^n 1 X) = X$ if either $n \in E$ or if $X(i) = 1$ for some $i$. In this case, for each $k\in\omega$, we have $\phi(0^n1 0^k) = 0^k$ if $n \in E$ and equals $\epsilon$ otherwise. 
\end{ex}

\begin{prop} \label{prop2} For any partial computable Turing functional $\Phi$, the canonical generator $\phi$ is computable in $\emptyset'$. 
\end{prop}

\begin{proof}  Let $\psi$ be some computable generator of $\Phi$.  Then for the canonical generator $\phi$, we have $\phi(\sigma) = \tau$ if and only if
\begin{itemize}
\item $(\exists n) (\forall \sigma' \in \{0,1\}^n)[ \sigma \prec \sigma' \implies \tau \preceq \psi(\sigma')]$, and

\item for $i=0,1$,  $\neg (\exists n) (\forall \sigma' \in \{0,1\}^n)[ \sigma \prec \sigma' \implies \tau^\frown i \preceq \psi(\sigma')]$.
\end{itemize}
Thus the graph of $\psi$ is a $\Sigma^0_2$ set and in fact is a difference of c.e.\ sets.
\end{proof}

Lastly, we can define the output/input ratio of a Turing functional given in terms of its canonical generator.

\begin{defn}
Let $\Phi$ be a partial Turing functional with canonical generator $\phi$. The $\Phi$-\emph{output/input ratio} given by $\sigma$, $\oi_\Phi(\sigma)$, is defined to be
\[
\oi_\Phi(\sigma)=\frac{|\phi(\sigma)|}{|\sigma|}.
\]
Similarly, for $X\in\cs$ we define $\oi_\Phi(X)$ to be
\[
\limsup_{n\rightarrow\infty}\frac{|\phi(X\uh n)|}{n}.
\]
We refer to $\oi_\Phi(X)$ as the \emph{$\Phi$-extraction rate along $X$}.
\end{defn}

%

\subsection{Average output/input ratios}

For a given generator $\phi$ of a Turing functional $\Phi$, we would like to define the \emph{average} $\phi$-output/input ratio.  However, such an average depends on an underlying probability measure on $\cs$.  Since we are interested, at least in part, in Turing functionals that extract unbiased randomness from biased random inputs, we need to consider average $\phi$-output/input ratios parametrized by an underlying measure. 


\begin{defn}
Given $\phi: \str \to \str$, the \emph{average $\phi$-output/input ratio} for strings of length $n$ with respect to $\mu$, denoted $\avg(\phi,\mu,n)$, is defined to be
\[\avg(\phi,\mu,n)=\sum_{\sigma\in2^n}\mu(\sigma)\oi_\phi(\sigma).\]
\end{defn}
\noindent Equivalently, we have
\[
\avg(\phi,\mu,n)=\frac{1}{n}\sum_{\sigma\in2^n}\mu(\sigma)|\phi(\sigma)|.
\]
Note that this is the $\mu$-average value of $\oi_\phi(X \uh n)$ over the space $\cs$, since this function is constant on each interval $\llb \sigma \rrb$. That is, if we fix $n$ and let $F_n(X) = \oi_{\phi}(X \uh n)$, then $F_n$ is a computable map from $\cs$ to $\mathbb{R}$ and the average value of $F_n$ on $\cs$ is given by 
\[
\int_{\cs} F_n(X) d\mu(X).
\]

We consider the behavior of this average in the limit, which leads to the following definition (which is adapted from one provided by Peres in \cite{Per92}).

\begin{defn}
For a function $\phi: \str \to \str$, the $\mu$-extraction rate of $\phi$, denoted $\rate(\phi,\mu)$, is defined to be
\[\rate(\phi,\mu)=\limsup_{n\rightarrow\infty}\avg(\phi,\mu,n).\]
In the case that $\phi$ is the canonical generator of a functional $\Phi$, we further define
\[
\rate(\Phi,\mu)=\rate(\phi,\mu).
\]
\end{defn}

%
%
%



\begin{ex} \label{ex3} Let $\Phi(X) = X \oplus X$, with a generator given by $\phi(\sigma) = \sigma \oplus \sigma$ (which, as noted above, is the canonical generator of $\Phi$). Then $|\phi(\sigma)| = 2 |\sigma|$ and hence $\oi_{\phi}(\sigma) = 2$ for all strings $\sigma$.  Thus the average output/input $\phi$ ratio is 2. Certainly $u_{\phi}(X,2n) = n$ but at the same time  
$u_{\phi}((\sigma \oplus \sigma) \fr i) = i+1$, so that $u_{\phi}(X,2n-1) = n$ and hence $\frac{2n-1}{u_{\phi}(X,2n-1)} = 2 - \frac 1n$.
Thus these rates agree in the limit but not at each level. Since $\oi_\phi(X \uh n) = 2$ for all $n$, we have the limit $\oi_{\phi}(X) = 2$ for all $X$ and hence the average output input ratio over all $X \in \cs$ is 
\[
\rate(\Phi,\mu) =\rate(\phi,\mu) = \int_{\cs} \oi_{\phi}(X)d\mu = 2,
\]
where the limit exists. Moreover, $\lim_{n\rightarrow\infty} \frac{n}{u_{\phi}(X,n)} = 2$ as well. 
\end{ex}

An interesting problem is to determine for which Turing functionals the $\limsup$ in the definition of extraction rate can be replaced with a limit. The following is an example where the limit does not exist.

\begin{ex}
Given a fixed function $\alpha: \omega \to \omega \setminus \{0\}$, define the total functional $\Phi_{\alpha}(X)$ for any input $X$ to be the infinite concatenation of the strings $X(i)^{\alpha(i)}$.  Thus if $\alpha(n) = 2$ for all $n$, then $\Phi_{\alpha}(X) = X \oplus X$. If $\alpha(n) =n+1$, then 
\[
\Phi_{\alpha}(X) = X(0)X(1)X(1)X(2)X(2)X(2)X(3)\dots.
\]  Now let $\alpha^*(n) = \sum_{i <n} \alpha(i)$, which is a strictly increasing function. It is clear that for \emph{any} strictly increasing function $\beta: \omega \to \omega$, there is a function $\alpha$ such that $\beta =\alpha^*$ and that, in general, $\alpha$ is computable if and only if $\alpha^*$ is computable. Fix $\alpha$ and $\beta = \alpha^*$ and let $\phi$ be the canonical generator of $\Phi_{\alpha}$.  Then we have $|\phi(X \uh n)| = \beta(n)$ for each $n\in\omega$, so that $\oi_\phi(\sigma) = \frac{\beta(n)}{n}$ for any string $\sigma$ of length $n$. Now the behavior of this limit is completely arbitrary.  For example, let $\beta(0) = 1$ and let $\beta(2^n+i) = 2^{n+1} + i$ for all $n$ and for all $i<2^n$.  Then $\oi_\phi(\sigma) = 2$ whenever $|\sigma|$ is a power of 2, but $\oi_\phi(\sigma) = \frac{2^{n+1}+i}{2^n+i}$ for $|\sigma| = 2^n+i$ and in particular, if $|\sigma| = 2^{n+1} -1$, then $\oi_\phi(\sigma) = \frac{3 \cdot 2^n - 1}{2 \cdot 2^n -1}$.  Thus $\limsup_{n\rightarrow\infty} \avg(\phi,\mu,n) = 2$ whereas   $\liminf_{n\rightarrow\infty} \avg(\phi,\mu,n) = 1.5$.
\end{ex}

For the $\lim_{n\rightarrow\infty}\avg(\phi,\mu,n)$ to exist,  the function $\phi$ must be regular in the relative amount of input needed for a given amount of output. The authors have studied some families of functions for which this is the case. First, there are the so-called \emph{online} continuous (or computable) functions, which compute exactly one bit of output for each bit of input (see \cite{CenRoj18}).  On the other hand, there are the \emph{random}
continuous functions which produce regularity in a probabilistic sense.  For example, the random continuous functions as defined by Barmpalias et al. \cite{BarBroCen08} produce outputs which are roughly $\frac 23$ as long, on the average, as the inputs.  See also \cite{CenPor15}.

Another example for which the limsup in the definition of rate is actually a limit is given by the following result.

\begin{lem}\label{lem-bct}
Suppose there exists some $c\in\omega$ such that $|\phi(\sigma)|\leq c|\sigma|$ for all $\sigma\in\str$ and that  there is some $r\in\mathbb{R}$ such that
\[
\lim_{n\rightarrow\infty}\dfrac{|\phi(X\uh n)|}{n}=r
\]
for $\mu$-almost every $X\in\cs$.  Then $\rate(\phi,\mu)=r$.
\end{lem}

\begin{proof}

Since there is some $c$ such that $\dfrac{|\phi(X\uh n)|}{n}\leq c$ for all $X\in\cs$, by the dominated convergence theorem, 
\begin{align*}
r=\int_{\cs} \lim_{n\rightarrow\infty}\dfrac{|\phi(X\uh n)|}{n}d\mu(X)&=\lim_{n\rightarrow\infty}\int_{\cs} \dfrac{|\phi(X\uh n)|}{n}d\mu(X)=\\
&=\lim_{n\rightarrow\infty}\avg(\phi,\mu,n)=\rate(\phi,\mu).
\end{align*}
\end{proof}

 In the next three sections, we consider several examples of Turing functionals $\Phi$ given by generators $\phi$ for which the following two conditions hold:
 \begin{itemize}
 \item[(i)] $\lim_{n\rightarrow\infty}\avg(\phi,\mu,n)$ exists (for an appropriate choice of the measure $\mu$), and 
 \item[(ii)] $\oi_\phi(X)=\lim_{n\rightarrow\infty}\oi_\phi(X\uh n)=\rate(\phi,\mu)$ for all sufficiently $\mu$-random sequences $X$.
 \end{itemize}
 That is, the extraction rate of $\phi$ is attained along sufficiently random inputs of $\Phi$.


\section{The Rate of Block Functionals}\label{sec-block}

For fixed $n\in\omega$, an \emph{$n$-block map} is a function $\phi: \str\rightarrow \str$ satisfying the following property: Given $\sigma\in\str$, we first write $\sigma={\sigma_1}^\frown\dotsc^\frown{\sigma_k}^\frown\tau$, where $|\sigma_i|=n$ for $i=1,\dotsc,k$ and $|\tau|<k$.  Then we have
\[
\phi(\sigma)=\phi(\sigma_1)^\frown\dotsc^\frown\phi(\sigma_k).
\]
That is, the behavior of $\phi$ is completely determined by its values of strings of length $n$ (and is undefined on all strings of length $k<n$).
An \emph{$n$-block functional} is a Turing functional $\Phi$ that has an $n$-block map $\phi$ as its canonical generator.  In this case we refer to $\phi$ as the $n$-block map associated to $\Phi$.  (Note that every $n$-block map can be extended to an $nk$-block map for $k\in\omega$ that induces the same functional.  Thus, the requirement that an $n$-block functional has an $n$-block map as a canonical generator ensures that an $n$-block functional isn't also an $nk$-block functional for every $k\in\omega$.)
We say that an $n$-block map $\phi:\str\rightarrow \str$ is \emph{non-trivial} if $|\phi(\sigma)|>0$ for some $\sigma\in 2^n$.

Block maps show up in the literature on randomness extraction, where typically one attempts to extract a sequence of unbiased random bits from a biased source.  For example, the 2-block map $\phi:\str\rightarrow\str$ defined by setting 
\begin{itemize}
\item $\phi(10)=0$,
\item $\phi(01)=1$, and
\item $\phi(00)=\phi(11)=\epsilon$
\end{itemize}
is precisely von Neumann's procedure.  Other examples of block maps in the randomness extraction literature are the randomizing functions studied by Elias in \cite{Eli72}, the iterations of von Neumann's procedure studied by Peres in \cite{Per92}, and extracting procedures studied by Pae in \cite{Pae16}.

We will determine the extraction rate of a $n$-block function with respect to a certain class of measures.  An \emph{$n$-step Bernoulli measure} is a Bernoulli measure on $(2^n)^\omega$. That is, an $n$-step Bernoulli measure is obtained by taking an infinite product of copies of some fixed measure on the set $2^n$.  Clearly, an $n$-step Bernoulli measure extends naturally to a measure on $\cs$.  Hereafter, we will use the term \emph{$n$-step Bernoulli measure} to refer to this extension.  Recall that a measure $\mu$ on $\cs$ is \emph{positive} if $\mu(\sigma)>0$ for all $\sigma\in\str$.

\begin{prop}\label{prop-at}
Let $n\in\omega$.  Suppose that $\mu$ is a positive $n$-step Bernoulli measure on $\cs$ and $\phi: \str\rightarrow\str$ is a non-trivial $n$-block map with associated $n$-block functional $\Phi$.  Then $\Phi$ is $\mu$-almost total.  
\end{prop}

\begin{proof}
Let $S=\{\tau\in 2^n\colon \phi(\tau)=\epsilon\}$, which is not equal to $2^n$ since $\phi$ is non-trivial.  Let $\U=S^\omega$, the set of all infinite sequences built up by concatenating members of $S$.  Since $\mu$ is positive and $S\neq 2^n$, $\sum_{\tau\in S} \mu(\tau)<1$, from which it follows that $\mu(\U)=0$.  Next, for each $\sigma\in(2^n)^{<\omega}$ such that $|\sigma|=nk$ for some $k\in\omega$, let $\U_\sigma=\{\sigma^\frown X\colon X\in \U\}$.  Clearly $\mu(\U_\sigma)=\mu(\sigma)\cdot\mu(\U)=0$, since $\mu$ is an $n$-step Bernoulli measure.  Then $\dom(\Phi)=\cs\setminus \bigcup_{\sigma\in (2^n)^{<\omega}}\U_\sigma$, from which it follows that $\mu(\dom(\Phi))=1$.
\end{proof}

\begin{thm}\label{thm-rate}
Let $\mu$ be a positive $n$-step Bernoulli measure on $\cs$ and $\phi: 2^n\rightarrow\str$ a non-trivial $n$-block map with associated $n$-block functional $\Phi$.  Then 
\[
\rate(\Phi,\mu)=\rate(\phi,\mu)=\avg(\phi,\mu,n)
\]
\end{thm}

\begin{proof}
We first note that if we consider the bits $\tau(0),\dotsc,\tau(n-1)$ of $\tau\in 2^n$ as a sequence of random variables, then  the expected value of $|\phi(\tau)|$ is
\[
E\bigl(|\phi(\tau(0),\dotsc,\tau(n-1))|\bigr)=\sum_{\sigma\in 2^n}\mu(\sigma)|\phi(\sigma)|,
\]
from which it follows that 
\[
\avg(\phi,\mu,n)=\dfrac{1}{n}E\bigl(|\phi(\tau(0),\dotsc,\tau(n-1))|\bigr).
\]

For $k\in\omega$, given a string of length $\tau={\tau_1}^\frown\cdots^\frown\tau_k$ of length $nk$ (where $|\tau_i|=n$ for $i=1,\dotsc,k$), since $\mu$ is an $n$-step Bernoulli measure, the blocks $\tau_1,\dotsc,\tau_k$ are independent.  Thus, the $\mu$-expected number of output bits for a string of length $nk$ is
\begin{align*}
E\bigl(|\phi(\tau(0),\dotsc,\tau(nk-1))|\bigr)&=\sum_{i=1}^kE\bigl(|\phi(\tau_i(0),\dotsc,\tau_i(n-1))|)\\
&=\sum_{i=1}^kn\cdot\avg(\phi,\mu,n)=nk\cdot\avg(\phi,\mu,n).
\end{align*}
Thus
\[
\avg(\phi,\mu,nk)=\frac{1}{nk}E\bigl(|\phi(\tau(0),\dotsc,\tau(nk-1))|\bigr)=\avg(\phi,\mu,n).
\]
For $k\in\omega$ and $i<n$, we have $\avg(\phi,\mu,nk+i)\leq \avg(\phi,\mu, nk)$, since the expected number of output bits of strings for inputs of length $nk+i$ is equal to the expected number of output bits for inputs of length $nk$.  It thus follows that 
\[
\rate(\Phi,\mu)=\rate(\phi,\mu)=\limsup_{k\rightarrow\infty}\avg(\phi,\mu,k)=\lim_{k\rightarrow\infty}\avg(\phi,\mu,nk)=\avg(\phi,\mu,n).
\]

\end{proof}

\begin{thm}\label{thm-nblock}
  Given $n\in\omega$, let $\mu$ be a computable, positive $n$-step Bernoulli measure on $\cs$, and let $X\in\cs$ be $\mu$-Schnorr random.  Then for every  non-trivial $n$-block map $\phi:\str\rightarrow\str$ with associated $n$-block functional $\Phi$,
\[
\oi_\Phi(X)=\rate(\Phi,\mu).
\]
\end{thm}

To prove Theorem \ref{thm-nblock}, we first need to develop some background.  Let $T:\cs\rightarrow\cs$ be the $n$-shift operator; that is, for $X\in\cs$ and $\sigma\in 2^n$, $T(\sigma^\frown X)=X$. For an $n$-step Bernoulli measure $\mu$ on $\cs$, $T$ is $\mu$-invariant, i.e., for any $\tau\in\str$, $\mu(\tau)=\mu(T^{-1}(\llb\tau\rrb))$.  Indeed, for any cylinder $\llb\tau\rrb$, 
\[
T^{-1}(\llb\tau\rrb)=\bigcup\llb\{\sigma\tau\colon \sigma\in 2^n\}\rrb.
\]
Thus,
\[
\mu\left(T^{-1}(\llb\tau\rrb)\right)=\sum_{\sigma\in2^n}\mu(\sigma\tau)=\mu(\tau)\sum_{\sigma\in 2^n}\mu(\sigma)=\mu(\tau).
\]
Recall that a $\mu$-invariant transformation $T:\cs\rightarrow \cs$ is ergodic if for any $\mathcal{A}\subseteq\cs$ such that $T^{-1}(\mathcal{A})=\mathcal{A}$, we have $\mu(\mathcal{A})=0$ or $\mu(\mathcal{A})=1$.  The following lemma is a useful characterization of ergodic transformations on $\cs$ (see \cite[Theorem 5.1.5, Theorem 6.3.4(1)]{Sil08}).

\begin{lem}\label{lem-ergodic}
Let $\mu$ be a measure on $\cs$ and let $T:\cs\rightarrow\cs$ be $\mu$-invariant.  Then $T$ is ergodic if and only if 
\[
\lim_{n\rightarrow\infty}\frac{1}{n}\sum_{i=0}^{n-1}\mu(T^{-i}\llb\sigma\rrb\cap\llb\tau\rrb)=\mu(\sigma)\mu(\tau)
\]
for all $\sigma,\tau\in\str$.
\end{lem}

The following result is straightforward, but we include it here for the sake of completeness.

\begin{lem}
The $n$-shift on $\cs$ is ergodic with respect to an $n$-step Bernoulli measure.
\end{lem}

\begin{proof}
We apply Lemma \ref{lem-ergodic}.  Let $\sigma,\tau\in\str$ be given.  Then there is some $k\in\omega$ and $m$ with $0\leq m< n$ such that $|\tau|=nk+m$.  Then for $j=2^{n-m}$, there are strings $\tau_1,\dotsc\tau_j$ of length $n(k+1)$ such that $\llb\tau\rrb=\bigcup_{i=1}^j\llb\tau_i\rrb$.  Then $T^{-(k+1)}(\llb\sigma\rrb)=\bigcup\llb\{\rho\sigma\colon \rho\in 2^{n(k+1)}\}\rrb$.  Note that for $\rho\in 2^{n(k+1)}$, 
\begin{equation}\label{eq1}
\mu\left(T^{-(k+1)}(\llb\sigma\rrb)\cap\llb\rho\rrb\right)=\mu(\sigma)\mu(\rho)
\end{equation}
Then we have
\[
T^{-(k+1)}(\llb\sigma\rrb)\cap\llb\tau\rrb=\bigcup_{i=1}^j(T^{-(k+1)}\left(\llb\sigma\rrb)\cap\llb\tau_i\rrb\right)
\]
and hence
\begin{align*}
\mu\left(T^{-(k+1)}(\llb\sigma\rrb)\cap\llb\tau\rrb\right)&=\sum_{i=1}^j\mu\left(T^{-(k+1)}(\llb\sigma\rrb)\cap\llb\tau_i\rrb\right)\\
&=\sum_{i=1}^j\mu(\sigma)\mu(\tau_i) &&  \hfill \text{by (\ref{eq1})}\\
&=\mu(\sigma)\mu(\tau)
\end{align*}
A similar argument shows that $\mu\left(T^{-{k'}}\!(\llb\sigma\rrb)\cap\llb\tau\rrb\right)=\mu(\sigma)\mu(\tau)$ for all $k'\geq k+1$.  It follows that
\[
\lim_{n\rightarrow\infty}\frac{1}{n}\sum_{i=0}^{n-1}\mu(T^{-i}\llb\sigma\rrb\cap\llb\tau\rrb)=\mu(\sigma)\mu(\tau),
\]
and hence by Lemma \ref{lem-ergodic}, $T$ is ergodic.
\end{proof}

The last ingredient we will use in the proof of Theorem \ref{thm-nblock} is the following effective version of Birkhoff's Ergodic Theorem due to Franklin and Towsner:
\begin{thm}[Franklin-Towsner \cite{FraTow14}]\label{thm-ergodic1}
Let $\mu$ be a computable measure on $\cs$ and let $T:\cs\rightarrow\cs$ be a computable, $\mu$-invariant, ergodic transformation.  Then for any bounded computable function $F$ and any $\mu$-Schnorr random $X\in\cs$,
\[
\lim_{k\rightarrow\infty}\frac{1}{k}\sum_{i=0}^{k-1}F(T^i(X))=\int F d\mu.
\]
\end{thm}

\begin{proof}[Proof of Theorem \ref{thm-nblock}]
Let $X\in\SR_\mu$.  Given $n\in\omega$, let $\mu$ be an $n$-step Bernoulli measure on $\cs$ and let $T$ be the $n$-shift. We define $F(X)=\dfrac{|\phi(X\uh n)|}{n}$. Then
\[
\int F d\mu=\sum_{\sigma\in 2^n}\mu(\sigma)\dfrac{|\phi(\sigma)|}{n}=\avg(\phi,\mu,n)=\rate(\phi,\mu),
\]
where the last equality holds by Theorem \ref{thm-rate}.  Next, for any $\mu$-Schnorr random sequence $X\in\cs$,
\[
\frac{1}{k}\sum_{i=0}^{k-1}F(T^i(X))=\frac{1}{k}\sum_{i=0}^{k-1}\dfrac{|\phi(T^i(X)\uh n)|}{n}=\frac{1}{nk}\sum_{i=0}^{k-1}\bigl|\phi\bigl(X\uh [ni,n(i+1)\bigr)\bigr|=\frac{|\phi(X\uh nk)|}{nk},
\]
where the last equality follows from the fact that $\phi$ is an $n$-block map.  Then
\begin{align*}
\begin{split}
\oi_\Phi(X)=\lim_{n\rightarrow\infty}\frac{|\phi(X\uh n)|}{n}=&
\lim_{k\rightarrow\infty}\frac{|\phi(X\uh nk)|}{nk}=\lim_{k\rightarrow\infty}\frac{1}{k}\sum_{i=0}^{k-1}F(T^i(X))\\
&=\int F d\mu=\rate(\phi,\mu)=\rate(\Phi,\mu),
\end{split}
\end{align*}
where the third equality follows from Theorem \ref{thm-ergodic1}, as the function $F$ is bounded.

\end{proof}

\section{The Rate of Functionals Induced by DDG-Trees}\label{sec-ddg}

The next example we consider is given in terms of DDG-trees (\emph{discrete distribution generating trees}), first introduced by Knuth and Yao in \cite{KnuYao76}.  A DDG-tree is a tree $S\subseteq\str$ with terminal nodes that can be used with unbiased random bits to induce a discrete probability distribution on a set $A=\{a_1,\cdots,a_k\}$.  The terminal nodes of $S$, the set of which is denoted by $D(S)$, are labelled with values from $A$.  We define a labelling function $\ell_S:D(S)\rightarrow A$ such that for all $\tau\in D(S)$, $\ell_S(\tau)\in A$ is the label assigned to $\tau$.  To ensure we have a probability distribution on $A$, the labels on $S$ must satisfy the following condition:  For $i=1,\dotsc, k$, if we set
\[
p_i=\sum_{\ell_S(\tau)=a_i}2^{-|\tau|},
\]
then 
\[
\sum_{i=1}^k p_i =1.
\]
The distribution $\{p_1,p_2,\dotsc,p_k\}$ on $A$ is induced by the following process:  
\begin{itemize}
\item For each branching node in the tree, we use the toss of an unbiased coin to determine which direction we will take. 
\item If we arrive at a terminal node $\tau$, the process outputs $\ell_S(\tau)$.
\end{itemize}
A DDG-tree $T$ defines a function from $\cs$ to $A$ as follows:  For $X\in\cs$, the output determined by $X$ is the unique element $a\in A$ such that for some $n\in\omega$, $X\uh n$ is a terminal node in $T$ labelled with $a$, if it exists; otherwise, the output is the empty string $\epsilon$.  That is, we look for the first $n$ such that $X\uh n\in D(S)$, and if such an $n$ exists, we output the value $\ell_S(X\uh n)$.

Knuth and Yao define the average running time of randomness extraction by a DDG-tree $S$ to be
\[
\mathit{AvgRT}(S)=\sum_{i\in\omega}i\cdot\lambda(\llb D(S)\cap 2^i\rrb).
\]
That is, $\mathit{AvgRT}(S)$ is the average number of input bits needed to produce a single output bit.

Hereafter we will restrict our attention to computable DDG-trees, where a DDG-tree $S$ is computable if the set $D(S)$ is a computable set and the function $\ell_S: D(S)\rightarrow A$ is computable  (which together imply that the values $p_1,\dotsc, p_k$ assigned to members of $A$ are computable).

We can use a computable DDG-tree $S$ to define a Turing functional as follows.  First, for every $\sigma\in D(S)$, we set $\phi_S(\sigma)=\ell_S(\sigma)$.   Then for any $\sigma\in\str$, if $\sigma$ does not extend any $\tau\in D(S)$, then we set $\phi_S(\sigma)=\epsilon$.  However, if $\sigma$ 
 extends some $\tau\in D(S)$, then we can write $\sigma={\sigma_1}^\frown\dotsc^\frown\sigma_k$, where $\sigma_1,\dotsc,\sigma_{k-1}\in D(S)$ and $\sigma_k\not\in D(S)$ (and is possibly empty).  Note that this decomposition is unique, as $D(S)$ is prefix-free.  Then we set
\[
\phi_S(\sigma)=\phi_S(\sigma_1)^\frown\dotsc^\frown\phi_S(\sigma_{k-1})^\frown\phi_S(\sigma_k)=\ell_S(\sigma_1)^\frown\dotsc^\frown\ell_S(\sigma_{k-1})^\frown\epsilon.
\]

We next extend $\phi_S$ to a Turing functional $\Phi_S:\cs\rightarrow A^\omega$.  For $X\in\cs$, we define a possibly finite sequence $n_0,n_1,\dotsc$ inductively as follows:
\begin{itemize}
\item $n_0$ is the unique $n$ such that $\ell_S(X\uh n)\in D(S)$, if it exists; otherwise $n_0$ is undefined.

\item Suppose $n_0,\dotsc, n_k$ have been defined.  Then $n_{k+1}$ is the unique $n$ such that $\ell_T(X\uh [n_k,n))\in D(S)$; otherwise $n_{k+1}$ is undefined.
\end{itemize}
Hereafter we will refer to the sequence of strings  $(X\uh [n_k,n_{k+1}))_{k\in\omega}$ as the \emph{$S$-blocks} of $X$.

If, for a given $X\in\cs$, the corresponding infinite sequence $(n_i)_{i\in\omega}$ is defined, then we set
\[
\Phi_S(X)=\ell_S(X\uh n_0)^\frown\ell_S(X\uh [n_0,n_1))^\frown\cdots^\frown\ell_S(X\uh[n_k,n_{k+1}))^\frown\cdots
\]
In the case that the corresponding sequence of block lengths is finite, then $\Phi_S(X)$ is undefined.

The issue of determining the canonical generator of a Turing functional defined in terms of a DDG-tree is a delicate one.  Knuth and Yao spend a considerable portion of their study \cite{KnuYao76} on the identification of the DDG-tree that most efficiently induces a distribution on a set $A$ (as well as more general distributions), where this efficiency is given in terms of extraction rate.  Hereafter, we will restrict our attention to DDG-trees $S$ that are minimal with respect to extraction rate, which amounts to assuming that the corresponding map $\phi_S$ on $\str$ is the canonical generator of the associated Turing functional $\Phi_S$.  Let us refer to such DDG-trees as \emph{minimal} DDG-trees.

\begin{prop}
If $S$ is a computable DDG-tree, then the Turing functional $\Phi_S$ is almost total.
\end{prop}

\begin{proof}
First, observe that collection of cylinders determined by the elements of $D(S)$ yields a set of Lebesgue measure one.  Indeed, 
\[
\lambda(\llb D(S)\rrb)=\sum_{i=1}^k\sum_{\ell_S(\tau)=a_i}2^{-|\tau|}=\sum_{i=1}^kp_i=1.
\]
It then follows that the set $\mathcal{P}=\{X\in\cs:(\forall n)\,\phi_S(X\uh n)=\epsilon\}$ is a $\Pi^0_1$ class of Lebesgue measure zero.  As in the proof of Proposition \ref{prop-at}, if we set $\mathcal{P}_\sigma=\{\sigma^\frown X\colon X\in \mathcal{P}\}$, then we have  $\lambda(\P_\sigma)=\lambda(\sigma)\cdot\lambda(\P)=0$.  Then $\dom(\Phi_S)=\cs\setminus \bigcup_{\sigma\in (D(S))^{<\omega}}\P_\sigma$, and so we have $\lambda(\dom(\Phi_S))=1$.
\end{proof}

We would like to calculate the extraction rate for a Turing functional $\Phi_S$ induced by a minimal DDG-tree $S$.  To do so, we will first prove the following:

\begin{thm}\label{thm-DDG}
  Let $X\in\cs$ be Schnorr random.  Then for every computable, minimal DDG-tree $S$, we have
\[
\oi_{\Phi_S}(X)=\frac{1}{\mathit{AvgRT}(S)}.
\]
\end{thm}

To prove Theorem \ref{thm-DDG}, we would like to mimic the proof of Theorem \ref{thm-nblock}.  In particular, we need to find an appropriate effective version of Birkhoff's ergodic theorem to derive the result.  However, to do so, we need to define the appropriate measure-preserving transformation.   

\begin{defn}
Let $S\subseteq\str$ be a tree with $\lambda(\llb D(S)\rrb)=1$.  The \emph{tree-shift} $T_S:\cs\rightarrow\cs$ is defined by setting $T_S(X)= Y$, where $X=\sigma^\frown Y$ and $\sigma\in D(S)$.  Moreover, in the case that $X\uh n\notin D(S)$ for all $n\in\omega$, $T_S(X)$ is undefined.  
\end{defn}

Note that if $S$ is a computable DDG-tree, then the associated tree-shift $T_S$ is computable by an almost total Turing functional, as $T_S$ is defined on $\llb D(S)\rrb$.

\begin{lem}
If $S$ is a tree with $\lambda(\llb D(S)\rrb)=1$, then the tree-shift $T_S$ is $\lambda$-invariant and ergodic.
\end{lem}

\begin{proof}

First, we show $\lambda$-invariance.  For $\tau\in\str$, we have 
\[
T_S^{-1}(\llb\tau\rrb)=\bigcup\llb\{\rho\tau\colon \rho\in D(S)\}\rrb.  
\]
Then
\[
\lambda(T_S^{-1}(\llb\tau\rrb))=\sum_{\sigma\in D(S)}\lambda(\sigma\tau)=\sum_{\sigma\in D(S)}\lambda(\sigma)\lambda(\tau)=\lambda(\tau)\sum_{\sigma\in D(S)}\lambda(\sigma)=\lambda(\tau).
\]

Next, we prove that $T_S$ is ergodic.  Towards this end, we claim that for every $\sigma,\tau\in\str$ and $n\in\omega$, $\lambda\bigl(T_S^{-n}(\llb\sigma\rrb)\cap\llb\tau\rrb\bigr)=\lambda(\sigma)\lambda(\tau)$.  We show this by induction on $n$. For the case in which $n=1$, given $\sigma,\tau\in\str$, there is a prefix-free set $\{\tau_i\}_{i\in\omega}\subseteq D(S)$  such that $\bigcup_{i\in\omega}\llb\tau_i\rrb\subseteq\llb\tau\rrb$ and
\[
\lambda(\tau)=\sum_{i\in\omega}\lambda(\tau_i);
\]
that is, $\llb\tau\rrb=\bigcup_{i\in\omega}\llb\tau_i\rrb$ up to a set of $\lambda$-measure zero.  Then 
\[
T_S^{-1}(\llb\sigma\rrb)\cap\llb\tau\rrb=\bigcup_{i\in\omega}\llb\tau_i\sigma\rrb
\]
and hence
\[
\lambda\bigl(T_S^{-1}(\llb\sigma\rrb)\cap\llb\tau\rrb\bigr)=\sum_{i\in\omega}\lambda(\tau_i\sigma)=\lambda(\sigma)\sum_{i\in\omega}\lambda(\tau_i)=\lambda(\sigma)\lambda(\tau).
\]
Next, suppose that $\lambda\bigl(T_S^{-n}(\llb\sigma\rrb)\cap\llb\tau\rrb\bigr)=\lambda(\sigma)\lambda(\tau)$.  Then 
\[
T_S^{-(n+1)}(\llb\sigma\rrb)=T_S^{-n}(T_S^{-1}(\llb\sigma\rrb))=T_S^{-n}\left(\bigcup\llb\{\rho\sigma\colon \rho\in D(S)\}\rrb\right)=\bigcup_{\rho\in D(S)}T_S^{-n}(\llb\rho\sigma\rrb).
\]
Then
\begin{align*}
\lambda\left(T_S^{-(n+1)}(\llb\sigma\rrb)\cap\llb\tau\rrb\right)&=\sum_{\rho\in D(S)}\lambda\left(T_S^{-n}(\llb\rho\sigma\rrb)\cap\llb\tau\rrb\right)\\
&=\sum_{\rho\in D(S)}\lambda(\rho\sigma)\lambda(\tau)
=\lambda(\sigma)\lambda(\tau)\sum_{\rho\in D(S)}\lambda(\rho)=\lambda(\sigma)\lambda(\tau),
\end{align*}
where the second equality follows from the inductive hypothesis.  It follows that
\[
\lim_{n\rightarrow\infty}\frac{1}{n}\sum_{i=0}^{n-1}\lambda(T_S^{-i}\llb\sigma\rrb\cap\llb\tau\rrb)=\lambda(\sigma)\lambda(\tau),
\]
and thus, by Lemma \ref{lem-ergodic}, $T_S$ is ergodic.
\end{proof}

The effective version of the ergodic theorem that we will use in the proof of Theorem \ref{thm-DDG} requires us to introduce some additional notions.  First, a function is \emph{a.e.\ computable} if it is computable on a $\Pi^0_2$ set of Lebesgue measure 1.   As noted above,  $T_S$ is computable on $\llb D(S)\rrb$, which is a $\Sigma^0_1$ class of measure 1, and so it is a.e.\ computable. 

Next, a function $F:\cs\rightarrow \mathbb{R}$ is \emph{effective integrable} (also \emph{$L^1$-computable}) if there is a computable sequence of rational step functions $(s_n)_{n\in\omega}$ such that $F(X)=\lim_{n\rightarrow\infty} s_n(X)$ (whenever $F(X)\halts$) and for all $n\in\omega$, $\int|s_n-s_{n-1}|d\lambda\leq 2^{-n}$; see, e.g. \cite{Rut20} or \cite{Miy13}.

We now can formulate the relevant effective version of Birkhoff's ergodic theorem, due to G\'{a}cs, Hoyrup, and Rojas \cite{GacHoyRoj11}  (as observed by Rute \cite{Rut20}, G\'{a}cs, Hoyrup, and Rojas prove a slightly different result, but the proof of their result establishes the following).

\begin{thm}[Effective Birkhoff's Ergodic Theorem, version 2 \cite{GacHoyRoj11} ]\label{thm-birk2}
Let $\mu$ be a computable measure on $\cs$ and let $T:\cs\rightarrow\cs$ be an a.e.\ computable, $\mu$-invariant, ergodic transformation.  Then for any a.e.\ computable function $F$ that is effectively integrable and any Schnorr random $X\in\cs$,
\[
\lim_{k\rightarrow\infty}\frac{1}{k}\sum_{i=0}^{k-1}F(T^i(X))=\int F d\mu.
\]
\end{thm}

\begin{proof}[Proof of Theorem \ref{thm-DDG}]
Let $X\in\cs$ be Schnorr random.  We define $F:\cs\rightarrow\omega$ so that $F(X)$ is the unique $n$ such that $X\uh n\in D(S)$; that is, $F$ counts the number of input bits of a given sequence $X$ needed to generate one bit of output using the DDG-tree $S$.  Clearly $F$ is also computable on $\llb D(S)\rrb$ and is thus a.e.\ computable.

To see that $f$ is effectively integrable, we define a sequence $s_n$ of rational step functions on $\cs$ as follows:
\[
s_n(X)=
\left\{
	\begin{array}{ll}
		 F(X) & \mbox{if } (\exists k\leq n) F(X)\halts=k \\
		n & \mbox{otherwise }
	\end{array}.
\right.
\]
Observe that $s_{n+1}(X)\leq s_n(X)+1$, and $s_n(X)=s_{n+1}(X)$ if and only if there is some $k\leq n$ such that $X\uh k\in D(S)$.  For $n\in\omega$, let us set $D(S)\uh n=\{\sigma\in D(S)\colon |\sigma|\leq n\}$.
\begin{align*}
\int |s_n-s_{n-1}|d\lambda&=\int s_n-s_{n-1} \,d\lambda\\
&=1\cdot\lambda\bigl(\cs\setminus \llb D(S)\uh (n-1)\rrb\}\bigr)+ 0\cdot\lambda\bigl( \llb D(S)\uh (n-1)\rrb\}\bigr)\\
&=\lambda\bigl(\cs\setminus \llb D(S)\uh (n-1)\rrb\}\bigr).
\end{align*}

Since $\lambda(\llb D(S)\rrb)=1$ and $D(S)$ is a computable set, the sequence  $\left(\lambda\bigl(\cs\setminus \llb D(S)\uh n\rrb\bigr)\right)_{n\in\omega}$ is uniformly computable and converges to 0.  Thus by choosing an appropriate subsequence of the functions $(s_n)_{n\in\omega}$, it follows that $F$ is effectively integrable.

Next, observe that
\begin{align}\label{eqn1}
\begin{split}
\int F d\lambda=\sum_{\sigma\in D(S)}\int_{\llb\sigma\rrb}F d\lambda&=\sum_{i\in\omega}i\cdot2^{-i}\cdot\bigl|\{\sigma\colon \sigma \in D(S)\cap 2^i\}\bigr|\\
&=\sum_{i\in\omega}i\cdot\lambda(\llb D(S)\cap 2^i\rrb)\\
&=\mathit{AvgRT}(S).
\end{split}
\end{align}

Given a Schnorr random $X\in\cs$, if we repeatedly apply the tree-shift $T_S$ to $X$ followed by the function $F$, we have
\begin{equation}\label{eqn2}
\sum_{i=0}^{k-1} F(T^i_S(X))=n_0+\sum_{i=0}^{k-1}\bigl|[n_i,n_{i+1})\bigr|=n_k
\end{equation}
where $(n_i)_{i\in\omega}$ is the sequence determined by the $S$-blocks of $X$.  Then it follows from (\ref{eqn2}) that
\begin{equation}\label{eqn3}
\lim_{k\rightarrow\infty}\frac{1}{k}\sum_{i=0}^{k-1} F(T^i_S(X))=\lim_{k\rightarrow\infty}\dfrac{n_k}{k}.
\end{equation}
Thus, 
\[
\lim_{k\rightarrow\infty}\dfrac{n_k}{k}=\lim_{k\rightarrow\infty}\frac{1}{k}\sum_{i=0}^{k-1} F(T^i_S(X))=\int F d\lambda=\mathit{AvgRT}(S),
\]
where the first equality is (\ref{eqn3}), the second equality follows from Theorem \ref{thm-birk2}, and the third equality comes from (\ref{eqn1}).

Lastly, we consider the values
\[
\dfrac{|\phi_S(X\uh n)|}{n}
\]
for $n\in\omega$.  Fix $n\in\omega$, if $(n_i)_{i\in\omega}$ is the sequence determined by the $S$-blocks of $X$, then for the maximum value $k$ such that $n_k\leq n$, 
\[
\dfrac{|\phi_S(X\uh n)|}{n}=\dfrac{k}{n}.
\]
Then
\[
\dfrac{k}{n_{k+1}}<\dfrac{|\phi_S(X\uh n)|}{n}\leq\dfrac{k}{n_k}.
\]
Since 
\[
\dfrac{k}{n_{k+1}}=\dfrac{k+1}{n_{k+1}}-\dfrac{1}{n_{k+1}},
\]
we have
\[
\lim_{k\rightarrow\infty}\left(\dfrac{k+1}{n_{k+1}}-\dfrac{1}{n_{k+1}}\right)<\lim_{n\rightarrow\infty}\dfrac{|\phi_S(X\uh n)|}{n}\leq\lim_{k\rightarrow\infty}\dfrac{k}{n_k}.
\]
It thus follows that
\[
\oi_{\Phi_S}(X)=\lim_{n\rightarrow\infty}\dfrac{|\phi_S(X\uh n)|}{n}=\lim_{k\rightarrow\infty}\dfrac{k}{n_k}=\dfrac{1}{AvgRT(S)}.
\]
\end{proof}

\begin{cor}\label{cor-rate}
$\rate(\Phi_S,\lambda)=\dfrac{1}{\mathit{AvgRT(S)}}$.
\end{cor}

\begin{proof}
We apply Lemma \ref{lem-bct}.  First, we have $\dfrac{|\phi_S(\sigma)|}{|\sigma|}\leq 1$ for every $\sigma\in\str$.  By Theorem \ref{thm-DDG}, 
\[
\lim_{n\rightarrow\infty}\dfrac{|\phi_S(X\uh n)|}{n}=\oi_{\Phi_S}(X)=\dfrac{1}{\mathit{AvgRT(S)}}
\]
for every Schnorr random sequence.  The conclusion immediately follows from Lemma \ref{lem-bct}.
\end{proof}

\section{The Extraction Rate of the Levin-Kautz Conversion Procedure}\label{sec-lk}

We now calculate the extraction rate of a general procedure due independently to Levin \cite{LevZvo70}, Kautz \cite{Kau91}, Schnorr and Fuchs in \cite{SchFuc77}, and Knuth and Yao \cite{KnuYao76}.  In addition, this procedure has been studied in the randomness extraction literature under the label of the \emph{interval algorithm} (see, for instance, \cite{HanHos97}) and is the main idea behind the data compression technique known as \emph{arithmetic coding} (see \cite[Chapter 4]{Say17}).

 Following \cite{BieMon12}, we will refer to this procedure as the \emph{Levin-Kautz conversion procedure}.  Here we prioritize Levin and Kautz, as they both used this procedure to study the conversion of Martin-L\"of random sequences with respect to one measure into Martin-L\"of random sequences with respect to another measure.  In particular, Levin and Kautz use the procedure to prove the following:

\begin{thm}[Levin \cite{LevZvo70}, Kautz \cite{Kau91}] \label{thm-lk}
For every computable $\mu,\nu\in\P(\cs)$, if $X\in\MLR_\mu$ and $X$ is not computable (and in particular, $\mu(\{X\})=0$), then there is some $Y\in\MLR_\nu$ such that $X\equiv_T Y$.
\end{thm}

We will give the basic idea of Levin-Kautz conversion procedure using the succinct approach due to Schnorr and Fuchs \cite{SchFuc77} in the context of converting biased randomness into unbiased randomness (i.e., randomness with respect to the Lebesgue measure).  For computable $\mu\in\P(\cs)$ and $\sigma\in\str$, we define two subintervals of $[0,1]$:
\[
(\sigma)_\lambda=\left[\sum_{i=1}^{|\sigma|}2^{-i}\sigma(i),\sum_{i=1}^{|\sigma|}2^{-i}\sigma(i)+2^{-|\sigma|}\right]
\]
and
\[
(\sigma)_\mu=\left[\sum_{\tau<_{\mathrm{lex}}\sigma;|\tau|=|\sigma|}\mu(\tau),\sum_{\tau\leq_{\mathrm{lex}}\sigma;|\tau|=|\sigma|}\mu(\tau)\right]
\]
(here $\leq_{\mathrm{lex}}$ defines the lexicographic ordering on strings of a fixed length). We define a Turing functional $\Phi_{\mu\rightarrow\lambda}$ as follows:  For $\sigma,\tau\in\str$, enumerate $(\sigma,\tau)$ into $S_{\Phi_{\mu\rightarrow\lambda}}$ if $(\sigma)_\mu\subseteq(\tau)_\lambda$.  Thus, given a $\mu$-random sequence $X$ as input, $\Phi_{\mu\rightarrow\lambda}$ treats $X$ as the representation of some $r\in[0,1]$ the bit values of which are determined by the values of the measure $\mu$.  For instance, the first bit of $r$ is a 0 if $r<\mu(0)$ and a 1 if $r>\mu(0)$ (and the procedure is undefined if $r=\mu(0)$).  $\Phi_{\mu\rightarrow\lambda}$ then outputs the standard binary generator of the real number $r$.   One can verify that the resulting Turing functional $\Phi_{\mu\rightarrow\lambda}$ is $\mu$-almost total and induces the Lebesgue measure on $\cs$.  

More generally, for computable measures $\mu,\nu$, one can similarly define an almost total functional $\Phi_{\mu\rightarrow\nu}$ that transforms $\mu$-randomness into $\nu$-randomness. Moreover, one can verify that, for non-computable $X\in\MLR_\mu$ and $Y\in\MLR_\nu$ such that $\Phi_{\mu\rightarrow\nu}(X)=Y$,
\begin{itemize}
\item[(i)] $(\Phi_{\mu\rightarrow\nu}\circ\Phi_{\nu\rightarrow\mu})(X)=X$, and
\item[(ii)] $(\Phi_{\nu\rightarrow\mu}\circ\Phi_{\mu\rightarrow\nu})(Y)=Y$.
\end{itemize}
Thus, given such a pair $X$ and $Y$, we clearly have $X\equiv_TY$.

We will consider this result in the context of strongly positive measures.  A measure $\mu$ on $\cs$ is \emph{strongly positive} if there is some $\delta\in(0,\frac{1}{2})$ such that
for every $\sigma\in\str$, $\mu(\sigma0\mid\sigma)\in[\delta,1-\delta]$; that is, all of the conditional probabilities associated with $\mu$ are bounded away from 0 and 1 by a fixed distance.  The main theorem we will prove in this section is an effective, pointwise version of a result due to Uyematsu and Kanaya \cite{UyeKan00}, who studied the extraction rate of the interval algorithm with respect to a general class of measures, namely the ergodic, shift-invariant ergodic measures.  Recall that for an ergodic, shift-invariant measure $\mu$ on $\cs$, the entropy of $\mu$ is defined to be
\[
h(\mu)=\lim_{n\rightarrow\infty}-\frac{1}{n}\sum_{|\sigma|=n}\mu(\sigma)\log\mu(\sigma).
\]

\begin{thm}\label{thm-lkrate}
Let $\mu$ and $\nu$ be computable, shift-invariant, ergodic measures that are strongly positive.  Then for every non-computable $A\in\MLR_\mu$, 
\[
\oi_{\Phi_{\mu\rightarrow\nu}}(A)=\frac{h(\mu)}{h(\nu)}.
\]
  In particular, in the case that $\nu=\lambda$, we have $\oi_{\Phi_{\mu\rightarrow\lambda}}(A)= h(\mu)$.
\end{thm}

Several remarks are in order.  First, by the Shannon source coding theorem \cite[Section 5.10]{CovTho12}, $\frac{h(\mu)}{h(\nu)}$ is the optimal rate for converting between $\mu$-randomness and $\nu$-randomness.  Second, Han and Hoshi \cite{HanHos97} showed that in the case that $\mu$ and $\nu$ are Bernoulli measures, $\rate(\Phi_{\mu\rightarrow\nu},\mu)=\frac{h(\mu)}{h(\nu)}$, but in the case that $\mu$ and $\nu$ are shift-invariant and ergodic, this is appears to be open (see \cite[Remark 14]{WatHan19}).

As a first step towards proving Theorem \ref{thm-lkrate}, we define an auxiliary function. Given $A\in\MLR_\mu$ and $B=\Phi_{\mu\rightarrow\nu}(A)$, let $\phi_{\mu\rightarrow\nu}$ be the canonical generator of $\Phi_{\mu\rightarrow\nu}$ and set
\[
g(n)=\max\{k\colon \phi_{\mu\rightarrow\nu}(A\uh n)\succeq B\uh k\}.
\]
Equivalently, $g(n)$ is the maximum value $k$ such that $(A\uh n)_\mu\subseteq(B\uh k)_\nu$.   It follows that the $\Phi_{\mu\rightarrow\nu}$-extraction rate of the computation $\Phi_{\mu\rightarrow\nu}(A)=B$ is $\oi_{\Phi_{\mu\rightarrow\nu}}(A)=\limsup_{n\rightarrow\infty}\frac{g(n)}{n}$.

%
%

We now calculate $\oi_{\Phi_{\mu\rightarrow\nu}}(A)$ for each $A\in\MLR_\mu$. We will make use of two additional results.  First, we use the following lemma due to Kautz:

\begin{lem}[Kautz \cite{Kau97}]\label{lem-kautz}
Suppose that $\mu$ and $\nu$ are computable and strongly positive, and let $\delta\in (0,\frac{1}{2})$ satisfy the condition that $\mu(\sigma 0\mid\sigma),\nu(\sigma 0\mid\sigma)\in[\delta,1-\delta]$ for every $\sigma\in\str$.  Suppose further that for $A,B\in\cs$ we have $\Phi_{\mu\rightarrow\nu}(A)=B$.  
\begin{itemize}
\item[(i)] For every $n\in\omega$, 
\[
\mu(A\uh n)\leq \nu(B\uh g(n)).
\]
\item[(ii)]There exists infinitely many $n\in\omega$ such that
\[
\delta^2\cdot\nu(B\uh g(n))\leq \mu(A\uh n).
\]

\end{itemize}
\end{lem}

Next, we use an effective version of the Shannon-McMillan-Breimann theorem due to Hoyrup \cite{Hoy12}.

\begin{thm}[Hoyrup \cite{Hoy12}]\label{thm-hoyrup}
Let $\mu$ be a computable, shift-invariant, ergodic measure on $\cs$.  Then for every $\mu$-Martin-L\"of random sequence $X\in\cs$,
\[
\lim_{n\rightarrow\infty}\frac{K(X\uh n)}{n}=\lim_{n\rightarrow\infty}\frac{-\log\mu(X\uh n)}{n}=h(\mu).
\]
\end{thm}
With these pieces, we now turn to the proof of our theorem.

\begin{proof}[Proof of Theorem \ref{thm-lkrate}]
Let $\Phi_{\mu\rightarrow\nu}(A)=B$ for $A\in\MLR_\mu$.  By Theorem \ref{thm-lk}, we have $B\in\MLR_\nu$. Since $\mu$ and $\nu$ are strongly positive, choose $\delta\in(0,\frac{1}{2})$ such that $\mu(\sigma 0\mid\sigma),\nu(\sigma0\mid\sigma)\in[\delta,1-\delta]$ for every $\sigma\in\str$.  By part (i) of Lemma \ref{lem-kautz}, we have
\[
\mu(A\uh n)\leq \nu(B\uh g(n))
\]
for all $n\in\omega$.  Applying the negative logarithm to both sides and dividing through by $n$ yields
\[
\frac{-\log\nu(B\uh g(n))}{n}\leq \frac{-\log\mu(A\uh n)}{n}
\]
for all $n\in\omega$.  It thus follows that 
\begin{equation}\label{lk-eq}
\limsup_{n\rightarrow\infty}\frac{-\log\nu(B\uh g(n))}{n}\leq \limsup_{n\rightarrow\infty}\frac{-\log\mu(A\uh n)}{n}
\end{equation}
for all $n\in\omega$.  Next, by part (ii) of Lemma \ref{lem-kautz}, we have
\[
\delta^2\cdot\nu(B\uh g(n))\leq \mu(A\uh n)
\]
for infinitely many $n\in\omega$.  Again, applying the negative logarithm to both sides and dividing through by $n$ yields, for some $c\in\omega$
\[
\frac{-\log\mu(A\uh n)}{n}\leq \frac{-\log\nu(B\uh g(n))+c}{n}.
\]
for infinitely $n\in\omega$.  It thus follows that
\begin{align*}
\liminf_{n\rightarrow\infty}\frac{-\log\mu(A\uh n)}{n}&\leq \limsup_{n\rightarrow\infty}\frac{-\log\nu(B\uh g(n))+c}{n}\\
&\leq \limsup_{n\rightarrow\infty}\frac{-\log\nu(B\uh g(n))}{n}+\limsup_{n\rightarrow\infty}\frac{c}{n}\\
&=\limsup_{n\rightarrow\infty}\frac{-\log\nu(B\uh g(n))}{n}.
\end{align*}
Combining this inequality with (\ref{lk-eq}), we thus have
\begin{equation}\label{lk-eq2}
\liminf_{n\rightarrow\infty}\frac{-\log\mu(A\uh n)}{n}\leq\limsup_{n\rightarrow\infty}\frac{-\log\nu(B\uh g(n))}{n}\leq \limsup_{n\rightarrow\infty}\frac{-\log\mu(A\uh n)}{n}.
\end{equation}
By Theorem \ref{thm-hoyrup}, since $A\in\MLR_\mu$, it follows from our assumptions on $\mu$ that 
\begin{equation}\label{lk-eq3}
\liminf_{n\rightarrow\infty}\frac{-\log\mu(A\uh n)}{n}=\limsup_{n\rightarrow\infty}\frac{-\log\mu(A\uh n)}{n}=\lim_{n\rightarrow\infty}\frac{-\log\mu(A\uh n)}{n}=h(\mu).
\end{equation}
Combining (\ref{lk-eq2}) and (\ref{lk-eq3}), we conclude
\begin{equation}\label{lk-eq3.5}
\limsup_{n\rightarrow\infty}\frac{-\log\nu(B\uh g(n))}{n}=h(\mu).
\end{equation}

Next, we use the fact that for positive sequences $(a_n)_{n\in\omega}$ and $(b_n)_{n\in\omega}$ such that $\lim_{n\rightarrow\infty}a_n$ exists,
\[
\limsup_{n\rightarrow\infty}(a_n\cdot b_n)=(\limsup_{n\rightarrow\infty}a_n)(\limsup_{n\rightarrow\infty}b_n)
\]
along with Equation (\ref{lk-eq3.5}) and  the fact that 
\[
\lim_{n\rightarrow\infty}\frac{-\log\nu(B\uh g(n))}{g(n)}=h(\nu),
\]
which follows from Theorem \ref{thm-hoyrup}, to derive the following:

\begin{align*}\label{lk-eq4}
\begin{split}
h(\mu)=\limsup_{n\rightarrow\infty}\frac{-\log\nu(B\uh g(n))}{n}&=\limsup_{n\rightarrow\infty}\left(\frac{-\log\nu(B\uh g(n))}{g(n)}\frac{g(n)}{n}\right)\\
&=\limsup_{n\rightarrow\infty}\left(\frac{-\log\nu(B\uh g(n))}{g(n)}\right)\limsup_{n\rightarrow\infty}\left(\frac{g(n)}{n}\right)\\
&=\lim_{n\rightarrow\infty}\left(\frac{-\log\nu(B\uh g(n))}{g(n)}\right)\limsup_{n\rightarrow\infty}\left(\frac{g(n)}{n}\right)\\
&=h(\nu)\cdot\oi_{\Phi_{\mu\rightarrow\nu}}(A).
\end{split}
\end{align*}
From this we can conclude that 
\[
\oi_{\Phi_{\mu\rightarrow\nu}}(A)=\frac{h(\mu)}{h(\nu)}.\qedhere
\]
\end{proof}

As noted above after the statement of Theorem \ref{thm-lkrate}, determining $\rate(\Phi_{\mu\rightarrow\nu},\mu)$ appears to be an open question in the randomness extraction literature.  Essentially, this problem boils down to finding a uniform bound of the sequence of functions given by $\dfrac{|\phi_{\mu\rightarrow\nu}(X\uh n)|}{n}$ to apply the dominated convergence theorem to calculate $\rate(\Phi_{\mu\rightarrow\nu},\mu)$.

\section{Open questions}\label{sec-questions}

We conclude with several open questions.  First, there is a general question about generalizing the results from Sections \ref{sec-block}, \ref{sec-ddg}, and \ref{sec-lk} to apply to a broader class of Turing functionals:

\begin{Q}
What features of an almost total Turing functional $\Phi$ guarantee that 
\[
\oi_\Phi(X)=\rate(\Phi,\mu)
\]
for the appropriate choice of measure $\mu$ and all sufficiently $\mu$-random sequences $X$?
\end{Q}

Next, the results we have established in showing the level of randomness that is sufficient for a sequence to witness the extraction rate of a Turing functional do not tell us what level of randomness is \emph{necessary} for this to hold.  Thus we can ask:

\begin{Q}
For each of the classes of Turing functionals that we have discussed, what is the level of randomness necessary for a sequence to witness the associated extraction rate?
\end{Q}

\bibliographystyle{alpha}
\bibliography{random}

\end{document}